\documentclass[a4paper,10pt,leqno, english]{amsart}

\usepackage{hyperref} 
\usepackage[lite]{amsrefs} 
\usepackage{microtype} 
\usepackage{verbatim} 
\usepackage{color}  
\usepackage{amsmath,amssymb}
\usepackage[all]{xy}

\title[The Farrell-Jones conjecture for lattices in Lie groups]
{The Farrell-Jones conjecture for arbitrary lattices in virtually connected Lie groups}
\author{Holger Kammeyer}
\author{Wolfgang L\"uck}
\author{Henrik R\"uping}
\address{Mathematisches Institut\\ Universit\"at Bonn\\ Endenicher Allee 60\\ 53115 Bonn\\ Germany}
\email{kammeyer@math.uni-bonn.de}
\urladdr{http://www.math.uni-bonn.de/people/kammeyer/}
\email{wolfgang.lueck@him.uni-bonn.de}
\urladdr{http://www.him.uni-bonn.de/lueck/}
\email{henrik.rueping@hcm.uni-bonn.de}
\urladdr{http://www.math.uni-bonn.de/people/rueping/}
\date{January 2014}

\newtheorem{theorem}{Theorem}

\newtheorem{proposition}{Proposition}
\newtheorem{lemma}{Lemma}

\theoremstyle{remark}
\newtheorem{remark}{Remark}

\makeatletter
   \let\c@corollary=\c@theorem
   \let\c@proposition=\c@theorem
   \let\c@remark=\c@theorem
   \let\c@lemma=\c@theorem
\makeatother

\newcommand*{\MRref}[2]{\href{http://www.ams.org/mathscinet-getitem?mr=#1}{MR \textbf{#1}}}
\newcommand*{\arXiv}[1]{\href{http://www.arxiv.org/abs/#1}{arXiv:\textbf{#1}}}

\newcommand*{\Z}{\mathbb Z}
\newcommand*{\Q}{\mathbb Q}
\newcommand*{\R}{\mathbb R}
\newcommand*{\C}{\mathbb C}

\newcommand*{\Gg}{\mathbf G}

\newcounter{commentcounter}

\newcommand{\ignore}[1]{}

\begin{document}

\begin{abstract}  
We prove the $K$- and the $L$-theoretic Farrell-Jones conjecture 
with coefficients in additive categories  and with  finite wreath products 
for arbitrary lattices in virtually connected Lie groups. 
\end{abstract} 

\maketitle


\section*{Introduction}

\noindent The Farrell-Jones conjecture predicts the algebraic $K$-theory and $L$-theory of group
rings.  The original formulation can be found in Farrell-Jones~\cite{Farrell-Jones}.  We
will deal with the more general version with coefficients in additive categories and with
finite wreath products, see~\cites{Bartels-Lueck:add,Bartels-Reich(2007)}, and ~\cite{Bartels:glnz}*{Definition~6.1}. 
The relevance of the Farrell-Jones Conjecture comes from the fact that it implies many other prominent
conjectures, for instance the one due to Borel about topological rigidity of aspherical
closed manifolds, the one due to Novikov about the homotopy invariance of higher
signatures, and the one due to Kaplansky about the triviality of idempotents in group
rings with coefficients in a field of torsionfree groups. For an overview of the
Farrell-Jones conjecture and its consequences, see for
instance~\cites{Bartels-Lueck-Reich:appl,Lueck:GroupRings,Lueck-Reich:survey}.

Let \(\mathfrak{FJ}\) be the class of groups satisfying the K- and L-theoretic
Farrell--Jones conjecture with coefficients in additive categories and with finite wreath
products. Recently, the third named author~\cite{Rueping:GLQ} proved 
\(\textup{GL}(n, \Q) \in \mathfrak{FJ}\).  In this article we show that the theory of deformations and rigidity
of lattices in semisimple Lie groups due to Calabi, Vesentini and Weil allow the following
conclusion.  A topological group is called \emph{virtually connected} if it has finitely
many path components.  A discrete subgroup of a locally compact Hausdorff group is 
called a \emph{lattice} if the quotient space \(G/\Gamma\) has finite covolume with respect to the Haar measure of \(G\).

\begin{theorem}[Lattices in virtually connected Lie groups] \label{thm:mainthm} Let \(G\)
  be a virtually connected Lie group and let \(\Gamma \subset G\) be a lattice.

Then \(\Gamma\) lies in \(\mathfrak{FJ}\).
\end{theorem}

\noindent This extends a previous result of Bartels, Farrell and the second named 
author~\cite{Bartels-Farrell-Lueck:Cocompact} from the class of cocompact lattices to the class
of all lattices in virtually connected Lie groups.

Deligne--Mostow~\cite{Deligne-Mostow:Monodromy} have constructed non-cocompact lattices in
\(\textup{SU}(2,1)\) and \(\textup{SU}(3,1)\) which are neither hyperbolic, nor
\(\textup{CAT}(0)\), nor arithmetic, nor solvable (not even amenable).  So
Theorem~\ref{thm:mainthm} comprises groups for which the Farrell-Jones conjecture was a
priori unknown. Note that the operator-theoretic version of the Farrell-Jones Conjecture,
the Baum-Connes conjecture for the topological $K$-theory of the reduced group
$C^*$-algebra, is still open for many lattices in virtually connected Lie groups, for
instance for $SL(n,\Z)$ for $n \ge 3$. 

The in our view most general result about lattices will be proved in
Theorem~\ref{thm:mostgeneral}, where the virtually connected Lie group is replaced by a
second countable locally compact Hausdorff group.

This paper has been financially supported by the Leibniz-Award, granted by the Deutsche
Forschungsgemeinschaft, of the second named author.


\section{The status of the Farrell-Jones Conjecture}

\noindent The next result describes what is known about $\mathfrak{FJ}$.
We will frequently use some of the properties of $\mathfrak{FJ}$ listed below, no more knowledge 
about the Farrell-Jones Conjecture is required to understand the proofs in this paper.

\begin{theorem}[Status of the Farrell-Jones Conjecture]
\label{the:FJinh} The class of groups $\mathfrak{FJ}$ has the following properties:
\begin{enumerate}
\item \label{the:FJinh:CAT} It contains all hyperbolic groups and all CAT(0)-groups;
\item \label{the:FJinh:solv} It contains all solvable groups;
\item \label{the:FJinh:GL} It contains \(\textup{GL}(n, \Q)\) and \(\textup{GL}(n, F(t))\) for any finite field $F$;
\item \label{the:FJinh:arith} It contains all $S$-arithmetic groups;
\item \label{the:FJinh:cocomapct_lattices} It contains all cocompact lattices in virtually connected Lie groups;
\item \label{the:FJinh:low} It contains the fundamental group of any manifold of dimension $\le 3$;
\item \label{the:FJinh:prod} It is closed under direct and free products;
\item \label{the:FJinh:sub} It is closed under taking subgroups;
\item \label{the:FJinh:hom} If $f:G\rightarrow H$ is a group homomorphism 
such that $H$, \(\ker f\) and $f^{-1}(Z)$ lie in $\mathfrak{FJ}$ for every 
infinite cyclic subgroup $Z\subset H$, then $G$ lies in $\mathfrak{FJ}$;
\item \label{the:FJinh:over} If a finite index subgroup of a group $G$ lies in $\mathfrak{FJ}$, so does $G$;
\end{enumerate}
\end{theorem}
\begin{proof}
  See for
  instance~\cites{Bartels-Lueck:CAT(0),Bartels-Farrell-Lueck:Cocompact,Bartels-Lueck-Reich:hyp,
    Rueping:GLQ,Wegner:CAT(0),Wegner:Solvable} for proofs (without wreath products).  The
  version and the corresponding proofs with wreath products are just a slight modification
  of the version without wreath products,
  compare~\cite{Bartels-Farrell-Lueck:Cocompact}*{Remark~0.4}.
\end{proof}

Farrell and Jones~\cite{Farrell-Jones(1998)}*{Proposition~0.10} have proved the
$L$-theoretic Farrell-Jones Conjecture and the
$K$-theoretic Farrell-Jones Conjecture in dimensions $\le 1$, both with untwisted coefficients in $\Z$,
for fundamental groups of A-regular negatively curved complete Riemannian manifolds. 

A Riemannian manifold is \emph{A-regular} if for some nonnegative sequence \( A = (A)_i \) we have
\( |\nabla^i R| \le A_i \)
where the indices i vary over the natural numbers and \(\nabla^i R\) is the \(i\)-th covariant derivative
 of the curvature tensor.

 If the lattice is torsionfree, then the quotient of the symmetric space by the action
 of that lattice is an $A$-regular manifold.  We want to also include lattices with torsion and allow twisted coefficients and 
want to consider all degrees. This makes the use of the important inheritance properties of the general version possible.


\section{Some preliminaries about finitely generated discrete subgroups 
in linear algebraic groups defined over \texorpdfstring{$\Q$}{Q}}

\noindent A key ingredient in our proof is the following striking property of lattices whose first
cohomology with coefficients in the adjoint representation 
vanishes~\cite{Raghunathan:Discrete}*{Proposition~6.6 and Theorem~6.7, pp.~90--91}.

\begin{theorem} \label{thm:latticerational} Let \(\Gg\) be a linear algebraic group
  defined over \(\Q\).  If \(\Gamma \subset \Gg(\R)\) is a finitely generated discrete
  subgroup and \(H^1(\Gamma, \mathfrak{g}) = 0\), then there exists a number field
  \(\mathbb{F}\) and an element \(g \in \Gg(\R)\) such that \(g \Gamma g^{-1} \subset   \Gg(\mathbb{F})\).
\end{theorem}

 \noindent We want to give an idea as to why group cohomology decides about the possibility of
conjugating a lattice into the \(\mathbb{F}\)-rational points of an algebraic group.  For
more information, see~\cite{Vinberg-et-al:Discrete}*{Section 6}.  Let \(G = \Gg(\R)\) and
let \(\textup{Hom}(\Gamma, G)\) be the space of all homomorphisms from \(\Gamma\) to \(G\)
with the topology of pointwise convergence.  A base point \(u \in \textup{Hom}(\Gamma,G)\) 
is given by the inclusion \(u \colon \Gamma \hookrightarrow G\).  A
\emph{deformation} of the lattice \(\Gamma\) in \(G\) is a map 
\(\varphi \colon I\rightarrow \textup{Hom}(\Gamma, G)\), defined on some open interval \(I\) 
containing zero, such that \(\varphi_0 = u\) and such that \(\varphi(\gamma) \colon I \rightarrow G\)
is smooth for every \(\gamma \in \Gamma\).  Given a smooth path \(g \colon I \rightarrow G\) 
with \(g_0 = e \in G\), we obtain a deformation setting \(\varphi(g)_t = g_t u g_t^{-1}\).  
These deformations are always present, regardless of the specific group
\(\Gamma\) and its embedding \(u\) in \(G\).  Therefore the deformations \(\varphi(g)\)
are termed \emph{trivial}.  Right multiplication with an element \(h \in G\) defines a
self-diffeomorphism \(R(h)\) of \(G\).  We identify the Lie algebra \(\mathfrak{g}\) of
\(G\) with the tangent space \(T_e(G)\).  Then any deformation \(\varphi\) defines a
function \(c(\varphi) \colon \Gamma \rightarrow \mathfrak{g}\) setting
\[ 
 c(\varphi)(\gamma) = \textup{d} R(\gamma^{-1}) \left. \frac{\textup{d}}{\textup{d}t}
  \varphi(\gamma)_t \right|_{t = 0}.
\] 
One easily verifies that \(c(\varphi)\) is a
cocycle of \(\Gamma\) with values in \(\textup{Ad} \circ u \colon \Gamma \rightarrow \mathfrak{g}\).  
For a smooth path \(g \colon I \rightarrow G\) with \(g_0 = e\) let \(X_g
\in \mathfrak{g}\) be the velocity vector of \(g\) at \(t = 0\).  Then for the trivial
deformation \(\varphi(g)\) we obtain 
\(c(\varphi(g))(\gamma) = X_g - \textup{Ad} \circ u (\gamma) \,X_g\) 
which means \(c(\varphi(g))\) is a coboundary.  This lets one hope that
the condition \(H^1(\Gamma, \mathfrak{g}) = 0\) might imply that every deformation of
\(\Gamma\) in \(G\) is trivial.  Indeed, this can be proved using, amongst other things,
the implicit function theorem.

Let \(\gamma_1, \ldots, \gamma_n\) be a choice of generators of \(\Gamma\).  Then we
obtain an embedding \(\textup{Hom}(\Gamma, G) \rightarrow G^n\) by sending a homomorphism
\(r \in \textup{Hom}(\Gamma, G)\) to the \(n\)-tuple \((r(\gamma_1), \ldots,
r(\gamma_n))\).  We write each relation of \(\Gamma\) as a word \(w = w(\gamma_1, \ldots, \gamma_n)\) 
in the symbols \(\gamma_i\).  Allowing general elements \(g \in G\) to take
the place of the \(\gamma_i\), each relation \(w\) defines a morphism \(w \colon G^n
\rightarrow G\) of real affine varieties defined over \(\Q\).  The image of our embedding
\(\textup{Hom}(\Gamma, G) \rightarrow G^n\) is then given by \(\bigcap_w w^{-1}(e)\), the
intersection taken over all relations in \(\Gamma\).  Thus \(\textup{Hom}(\Gamma, G)\) is
embedded as a \(\Q\)-subvariety of \(G^n\), no matter whether finitely many relations are
sufficient or not.  If \(H^1(\Gamma, \mathfrak{g}) = 0\), then every deformation of
\(\Gamma\) in \(G\) is trivial.  So in that case the orbit of \(u\) in
\(\textup{Hom}(\Gamma, G)\) under the \(G\)-action by conjugation contains an open
neighborhood \(U \subset \textup{Hom}(\Gamma, G)\) of \(u\) which also implies that the
point \(u\) is simple.  It is then a lemma of algebraic 
geometry~\cite{Garland-Raghunathan:Fundamental}*{Lemma~7.1, p.~311} that \(U\) contains a
\(\overline{\Q}\)-rational point \(u'\) where \(\overline{\Q}\) is the algebraic closure
of \(\Q\).  Since \(u'\) is an \(n\)-tuple of matrices, clearly \(u'\) is in fact an
\(\mathbb{F}\)-rational point for a finite extension \(\mathbb{F}\) of \(\Q\).  This gives
the theorem.  It remains the question which lattices \(\Gamma\) have vanishing
\(H^1(\Gamma, \mathfrak{g})\).

Let \(G\) be a connected semisimple Lie group. The Lie algebra \(\mathfrak{g}\) of \(G\) has a decomposition
\(\mathfrak{g}=\mathfrak{g}_1 \oplus \cdots \oplus \mathfrak{g}_k\) into simple ideals
which is unique up to permutation.  The unique connected Lie subgroup \(G_i\) in
\(G\) with Lie subalgebra \(\mathfrak{g}_i\) is a priori not closed.  But since \(\mathfrak{g}_i\) is an ideal, the group \(G_i\) is normal and hence actually closed by~\cite{Ragozin:Normal}.  Multiplication defines an epimorphism 
\begin{eqnarray}
G_1 \times \cdots \times G_k & \rightarrow & G
\label{almost_direct_product}
\end{eqnarray}
with discrete (in fact central) kernel.  This is equivalent to \(G\) being the \emph{almost
  direct product} of the normal subgroups \(G_i\), i.e., $G = G_1 \cdot G_2 \cdot \cdots
\cdot G_k := \{g_1 \cdot g_2 \cdot \cdots \cdot g_k \mid g_i \in G_i, i = 1,2, \ldots
,k\}$ and the intersection of $G_i$ with $\prod_{j \not= i} G_j$ is discrete for all $i$.
A \emph{compact factor} \(K\) of \(G\) is a connected, normal,
compact subgroup of \(G\).  It follows from \cite{Ragozin:Normal} that \(K\) is an almost
direct product \(K = G_{i_1} \cdots G_{i_l}\) with \(1 \leq i_1 < \cdots < i_l \leq k\)
and each \(G_{i_j}\) compact.

Thus a \emph{connected semisimple Lie group without compact factors} is a connected
semisimple Lie group all of whose connected, normal, compact subgroups are trivial.  We
remark that there is an equivalent definition in the literature requiring instead that all
compact quotient groups of \(G\) are trivial.

 Following~\cite{Raghunathan:Discrete}*{Definition~5.20, p.~86},
we  call a lattice \(\Gamma \subset G\) 
\emph{irreducible} if there are no two normal, connected, infinite subgroups \(H_1, H_2
\subset G\) such that \(G\) is the almost direct product of \(H_1\) and \(H_2\) and such
that \((\Gamma \cap H_1)\cdot (\Gamma \cap H_2)\) has finite index in \(\Gamma\).

\begin{theorem} \label{thm:latticecohomology} Let \(G\) be a connected semisimple Lie
  group without compact factors and let \(\Gamma \subset G\) be an irreducible lattice.
  Suppose that \(G\) is not locally isomorphic to \(\textup{SL}(2, \R)\) or
  \(\textup{SL}(2, \C)\).  Then \(H^1(\Gamma, \mathfrak{g}) = 0\).
\end{theorem}
\begin{proof}
  Since the group \(G_i\) appearing in~\eqref{almost_direct_product} is compact if and
  only if \(G_i\) has real rank zero, we see that a rank one semisimple Lie group without
  compact factors is actually simple.  Therefore Theorem~\ref{thm:latticecohomology} is a
  combination of~\cite{Vinberg-et-al:Discrete}*{Corollary~7.5 and Theorem~7.7, p.~99}.
\end{proof}


\section{Proof of Theorem~\ref{thm:mainthm}}

\noindent We first prove the semisimple case of Theorem~\ref{thm:mainthm} and then show how to deduce the result in
general.

\begin{proposition}\label{prop:FJsemisimp}
  Let \(G\) be a connected semisimple Lie group and let \(\Gamma \subset G\) be a lattice.
  Then \(\Gamma \in \mathfrak{FJ}\).
\end{proposition}

\begin{proof}
  Let \(K \subset G\) be the maximal compact factor.  Consider the short exact sequences
  \[
  \xymatrix{1\ar[r]& K\ar[r]& G\ar[r]^p & G/K\ar[r]& 1\\
    1\ar[r]& \Gamma\cap K \ar[r]\ar@{^{(}->}[u]& \Gamma \ar[r]\ar@{^{(}->}[u]& p(\Gamma) \ar[r]\ar@{^{(}->}[u]& 1.}
   \]
  By~\cite{Vinberg-et-al:Discrete}*{Corollary~4.10, p.~24} the image \(p(\Gamma)\) of
  \(\Gamma\) is a lattice in \(G/K\).  The group \(\Gamma \cap K\) is finite because \(K\)
  is compact and \(\Gamma\) is discrete.  Thus any preimage of a virtually cyclic subgroup
  is again virtually cyclic and hence lies in \(\mathfrak{FJ}\), for example by
  Theorem~\ref{the:FJinh}~\eqref{the:FJinh:CAT}.  It follows from
  Theorem~\ref{the:FJinh}~\eqref{the:FJinh:hom} that the lattice \(\Gamma\) lies in
  \(\mathfrak{FJ}\) if \(p(\Gamma)\) does.  Thus we may assume that \(G\) has no compact
  factors.

  Let \(Z(G)\) denote the center of \(G\).  It follows now 
  from~\cite{Raghunathan:Discrete}*{Corollary~5.17, p.~84} that the product \(Z(G)\Gamma\) is
  discrete and in particular closed in \(G\) because \(G\) is a Hausdorff group.  Since
  the center \(Z(G)\) is a normal subgroup, we obtain 
from~\cite{Raghunathan:Discrete}*{Theorem~1.13, p.~23} that \(\Gamma \cap Z(G)\) is a lattice in
  \(Z(G)\), which in this case just means it has finite index.  
  Thus~\cite{Onishchik:EncyclopaediaI}*{Theorem~4.7, p.~23} says that \(\Gamma\) projects
  under \(p \colon G \rightarrow G/Z(G)\) to a lattice \(p(\Gamma) \subset G/Z(G)\).
  Moreover, the center \(Z(G)\) is an abelian group so that any preimage under \(p\) of a
  virtually cyclic subgroup in \(G/Z(G)\) is virtually solvable, thus lies in
  \(\mathfrak{FJ}\) by Theorem~\ref{the:FJinh}~\eqref{the:FJinh:solv} and~\eqref{the:FJinh:over}.   
  Again by  Theorem~\ref{the:FJinh}~\eqref{the:FJinh:hom} we may assume that \(G\) has trivial
  center.

  We conclude from~\cite{Raghunathan:Discrete}*{Theorem~5.22, p.~86} that we have an
  almost direct product decomposition \(G = H_1 \cdots H_r\) such that the almost direct
  product \(\Gamma_1 \cdots \Gamma_r\) has finite index in \(\Gamma\) where 
  \(\Gamma_i =   \Gamma \cap H_i\) is an irreducible lattice in \(H_i\) for each \(i = 1, \ldots, r\).
  Recall that discrete normal subgroups of connected topological groups are central.  Since \(G\) has 
trivial center, both almost direct products are actually direct products.  Thus by
  Theorem~\ref{the:FJinh}~\eqref{the:FJinh:prod} and~\eqref{the:FJinh:sub} we can assume
  that \(\Gamma\) is an irreducible lattice in a connected semisimple Lie group \(G\) with
  trivial center and without compact factors.

  Suppose \(G\) was locally isomorphic to \(\textup{SL}(2, \R)\) or \(\textup{SL}(2,
  \C)\).  Since \(G\) has trivial center, \(G\) is actually globally isomorphic to
  \(\textup{PSL}(2, \R)\) or \(\textup{PSL}(2, \C)\).  Thus \(\Gamma\) acts properly with
  finite volume quotient on hyperbolic 2- or 3-space.  
  Therefore~\cite{Bridson-Haefliger:MetricSpaces}*{Corollary~11.28, p.~362} asserts that \(\Gamma\)
  is \(\textup{CAT}(0)\) whence in \(\mathfrak{FJ}\) by
  Theorem~\ref{the:FJinh}~\eqref{the:FJinh:CAT}.  So we may assume that \(G\) is neither
  locally isomorphic to \(\textup{SL}(2, \R)\) nor to \(\textup{SL}(2, \C)\).  By
  Theorem~\ref{thm:latticecohomology} we then have \(H^1(\Gamma, \mathfrak{g}) = 0\).  Moreover, 
the lattice \(\Gamma \subset G\) is finitely generated, see~\cite{Raghunathan:Discrete}*{Remark~13.21, p.~210}.

  Let \(\mathfrak{g}\) be the Lie algebra of \(G\).  The adjoint representation
  \(\textup{Ad} \colon G \rightarrow \textup{Aut}(\mathfrak{g})\) embeds \(G\) as the Lie
  subgroup of inner automorphisms \(\textup{Int}(\mathfrak{g}) \subset
  \textup{Aut}(\mathfrak{g})\) as is shown in~\cite{Helgason:Symmetric}*{5.2.(ii),
    p.~129}.  Since \(\mathfrak{g}\) is semisimple, the subgroup
  \(\textup{Int}(\mathfrak{g})\) is actually just the identity component of
  \(\textup{Aut}(\mathfrak{g})\) by~\cite{Helgason:Symmetric}*{Corollary~6.5, p.~132}.
  In addition it is well-known that a real semisimple Lie algebra admits a basis with
  rational structure constants~\cite{Borel:CliffordKlein}*{Proposition~3.7, p.~118}.  It
  follows that \(\Gg = \textup{Aut}(\mathfrak{g}) \subset
  \textup{GL}(\mathfrak{g})\) is a linear algebraic \(\Q\)-group and 
  \(G =  \Gg(\R)^0\).

  Finally, Theorem~\ref{thm:latticerational} asserts that \(\Gamma\) is conjugate to a
  subgroup of the \(\mathbb{F}\)-rational points \(\Gg(\mathbb{F})\) for a number field
  \(\mathbb{F}\).  By restriction of scalars~\cite{Platonov-Rapinchuk:AlgebraicGroups}*{Section~2.1.2} there exists a linear
  algebraic \(\Q\)-group \(\textup{res}_{\mathbb{F}/\Q}(\Gg)\) such that \(\Gg(\mathbb{F})
  = \textup{res}_{\mathbb{F} / \Q}(\Gg)(\Q) \subset \textup{GL}(n, \Q)\).
  Theorem~\ref{the:FJinh}~\eqref{the:FJinh:GL} and \eqref{the:FJinh:sub} completes the proof.
\end{proof}

\begin{remark}
  The above proof starts by showing that we can assume the Lie group \(G\) is connected,
  semisimple, center-free and has no compact factors while the lattice \(\Gamma\) is
  irreducible.  If one additionally requires that \(G\) has real rank at least two, then
  the assumptions of Margulis' famous arithmeticity theorem
  \cite{Margulis:Arithmeticity}*{Theorem~1} are satisfied which says in particular that
  \(\Gamma\) virtually embeds into \(\textup{GL}(n, \Z)\).  The Farrell--Jones conjecture
  for \(\Gamma\) then follows from \cite{Bartels:glnz} with no more detour.  The existence
  of nonarithmetic lattices in real rank one Lie groups, however, necessitates our
  appealing to the more classical local rigidity theory.  The latter makes weaker
  assumptions on \(G\) at the cost of the weaker conclusion \(\Gamma \subset
  \textup{GL}(n, \Q)\).  But this turns out to be good enough for our purposes.
\end{remark}

\begin{proof}[Proof of Theorem~\ref{thm:mainthm}.]
  We will mostly follow~\cite{Bartels-Farrell-Lueck:Cocompact}*{Proof of Proposition~5.1,
    p.~38}. We will prove this by induction on the dimension of the surrounding Lie group
  $G$. If $G$ is zero-dimensional, $\Gamma$ is a finite group and thus trivially satisfies
  the Farrell--Jones conjecture. Since \(G\) has finitely many path components,
  \(\Gamma\cap G^0\) has finite index in \(\Gamma\) and is a lattice in \(G^0\) so that we
  may assume that \(G\) is connected by Theorem~\ref{the:FJinh}~\eqref{the:FJinh:over}.
  Arguing as in the first part of the proof of Proposition~\ref{prop:FJsemisimp}, 
  we may moreover assume that \(G\) has no connected, compact, normal subgroup.

  Let \(R\) be the \emph{radical} of \(G\) given by the maximal connected normal solvable
  subgroup of \(G\).  Similarly we denote by \(N\) the \emph{nilradical} of \(G\) given by
  the maximal connected normal nilpotent subgroup of \(G\).  Clearly \(N \trianglelefteq  R\) 
   and \(R/N\) is abelian.  Recall from~\cite{Varadarajan:LieGroups}*{Theorem~3.18.13,
    p.~244} that \(G\) possesses maximal semisimple subgroups, any such two are
  conjugate, and for any such \(S \subset G\) we have the \emph{generalized Levi decomposition} \(G = RS\). 
  As a word of warning, in general neither \(S\) nor the
  intersection \(R \cap S\) is a closed subgroup of \(G\). An example of such an \(S\) is given by Alain Valette in \cite{MOAnswer}.

  We want to prove that \(\Gamma \cap N\) is a lattice in \(N\).  According 
  to~\cite{Vinberg-et-al:Discrete}*{Theorem~1.6, p.~106} a sufficient condition is that
  every compact factor of \(S\) acts non-trivially on \(R\).  Suppose \(K\) was a compact
  factor in \(S\) acting trivially on \(R\).  Any element $g\in G$ is of the form $g=rs$
  with \(r \in R\) and \(s \in S\).  We get  for \(k\in K\) that $sks^{-1} \in K$ and $(sks^{-1})^{-1}r(sks^{-1})  = r$
  and hence \(gkg^{-1}=   r(sks^{-1})r^{-1}=sks^{-1} \in K\).  Thus \(K\) is normal in \(G\) whence trivial.  It
  follows that \(\Gamma \cap N\) is a lattice in \(N\).  
   By~\cite{Vinberg-et-al:Discrete}*{Theorem~4.7, p.~23\footnote{Note that Theorem~4.7 in
      loc.\,cit.\,is obviously misprinted.  The conclusion should read 
    ``\ldots if and  only if \(\Gamma \cap H\) is a lattice in \(H\).''}} \label{footnote} we conclude
  that \(\Gamma / (\Gamma \cap N)\) is a lattice in \(G/N\).  We have the short exact
  sequence
  \[ 
   1 \rightarrow \Gamma \cap N \rightarrow \Gamma\stackrel{p}{\rightarrow} \Gamma/
  (\Gamma \cap N) \rightarrow 1. 
   \] 
   Since \(G\) has no connected, compact, normal
  subgroup, the nilradical is simply 
  connected~\cite{Greenleaf-Moskowitz-Rothschild:UnboundedConjugacy}*{Lemma~3.1.(i), p.~229}.  
  It  follows from~\cite{Raghunathan:Discrete}*{Proposition~3.7, p.~52} that the lattice
  \(\Gamma \cap N \subset N\) is \emph{poly-\(\Z\)}, i.e., polycyclic with infinite cyclic
  factor groups.  Hence the preimage of a virtually cyclic subgroup of
  \(\Gamma/(\Gamma\cap N)\) under \(p\) is virtually poly-\(\Z\) as well, thus lies in
  \(\mathfrak{FJ}\) by Theorem~\ref{the:FJinh}~\eqref{the:FJinh:solv} and~\eqref{the:FJinh:over}. Now
  if \(N\) is non-trivial, then the Lie group \(G/N\) is of lower dimension than \(G\) and
  hence \(\Gamma / (\Gamma \cap N) \in \mathfrak{FJ}\) by the induction hypothesis.
  Therefore \(\Gamma \in \mathfrak{FJ}\) by
  Theorem~\ref{the:FJinh}~\eqref{the:FJinh:hom}.  If on the other hand \(N\) is trivial,
  then \(R \cong R/N\) is abelian, so \(R\) is contained in the nilradical \(N\) and thus
  trivial.  Therefore \(G\) is semisimple and Proposition~\ref{prop:FJsemisimp} completes
  the proof.
\end{proof}

\begin{remark}
  According to~\cite{Raghunathan:Discrete}*{Corollary~8.28, p.~150} the criterion that no
  compact factor of \(S\) acts trivially on \(R\) is actually sufficient for \(\Gamma /
  \Gamma \cap R\) being a lattice in \(G/R\).  This result would spare us the detour of
  factoring out the nilradical and using induction on \(\dim G\).  However,
  A.~N.~Starkov~\cite{Starkov:Counterexample} claims to construct a counterexample to
  Corollary~8.28, which earned him a doubtful Mathematical Review and a follow-up paper by
  T.~S.~Wu~\cite{Wu:Note} counterclaiming to give a new proof of the result in question.
  On the other hand, E.~B.~Vinberg et.~al.~\cite{Vinberg-et-al:Discrete}*{p.~107} say
  the counterexample of Starkov is correct.  We refrain from taking sides in the
  discussion and prefer to give our more involved but safe argument.
\end{remark}


\section{Generalizations}

\noindent Two assumptions in Theorem~\ref{thm:mainthm} can still be relaxed.  Firstly, the notion of
lattice still makes sense for locally compact Hausdorff groups because of the existence of
a unique Haar measure.  Secondly, one can try and work with less restrictive connectivity.
Here is the most general result we could come up with.

\begin{theorem}[Lattices in second countable locally compact Hausdorff groups]
  \label{thm:mostgeneral}
  Let $G$ be a second countable locally compact Hausdorff group and let $\Gamma$ be a
  lattice in $G$. If $\pi_0(G)$ is discrete and lies in $\mathfrak{FJ}$, then $\Gamma$
  also lies in $\mathfrak{FJ}$.
\end{theorem}

\noindent We remark that the Hausdorff assumption for topological groups is often implicit in the
literature.  In fact a topological \(T_0\)-group is already Hausdorff.  For the proof of
Theorem~\ref{thm:mostgeneral} let us first recall some basic facts about the automorphism group
\(\textup{Aut}(G)\) of a connected Lie group \(G\).  Differentiation defines a
homomorphism of Lie groups \(\textup{d} \colon \textup{Aut}(G) \rightarrow
\textup{Aut}(\mathfrak{g})\) which is actually injective and has closed 
image~\cite{Onishchik:EncyclopaediaI}*{Proposition~4.1, p.~49}.  If \(G\) is simply-connected
or has trivial center, then this map is an isomorphism.  For the inner automorphisms
\(\textup{Int}(G)\) of \(G\) we obtain \(\textup{d} (\textup{Int}(G)) = \textup{Ad}(G)\)
by the very definition of the adjoint representation.  As remarked in the proof of
Proposition~\ref{prop:FJsemisimp} we have \(\textup{Ad}(G) = \textup{Int}(\mathfrak{g})\)
so that \(\textup{d}\) induces an injective group homomorphism 
\(\textup{Out}(G) \rightarrow \textup{Out}(\mathfrak{g})\) where 
\(\textup{Out}(G) = \textup{Aut}(G) / \textup{Int}(G)\) and 
\(\textup{Out}(\mathfrak{g}) = \textup{Aut}(\mathfrak{g}) / \textup{Int}(\mathfrak{g})\) 
denote the groups of outer automorphisms.  If \(G\) is
moreover semisimple, then we have seen in the same proof that 
\(\Gg = \textup{Aut}(\mathfrak{g})\) is a linear algebraic \(\Q\)-group with
\(\Gg(\R)^0 = \textup{Int}(\mathfrak{g})\).  By a theorem of
Whitney~\cite{Whitney:Elementary}*{Theorem~3, p.~547} a real algebraic variety has only
finitely many components in the ordinary topology.  Applying this result to the
\(\R\)-variety \(\Gg(\R)\) we have come to the following conclusion. 

\begin{lemma} \label{lemma:OutG} The group of outer automorphisms \(\textup{Out}(G)\) of a
  connected semisimple Lie group \(G\) is finite.
\end{lemma}

\noindent We use this fact to draw the following conclusion.

\begin{proposition}\label{prop:FJconnauto}
  Let $\Gamma$ be a lattice in a connected Lie group $G$, and let $\varphi$ be an
  automorphism of $G$ with $\varphi(\Gamma)=\Gamma$. Then $\Gamma\rtimes_\varphi
  \mathbb{Z}$ lies in $\mathfrak{FJ}$.
\end{proposition}

\begin{proof}
  We will prove this by induction on the dimension of $G$. The induction beginning
  $\dim(G)=0$ is trivial; in this case $G$ and hence $\Gamma$ are trivial so that 
  $\Gamma\rtimes \mathbb{Z}\cong   \mathbb{Z}$ 
  and hence  satisfies the Farrell-Jones conjecture for trivial reasons.

  First we want to reduce the general case to the case of a lattice in a semisimple Lie
  group. This works similarly as in the proof of Theorem~\ref{thm:mainthm}.  The
  nilradical \(N\) of \(G\) is characteristic, therefore \(\varphi\)-invariant and hence
  we get a short exact sequence.
  \[ 
   1\rightarrow \Gamma \cap N\rightarrow \Gamma \rtimes_\varphi \mathbb{Z}\rightarrow
  \Gamma/(\Gamma\cap N)\rtimes \mathbb{Z}\rightarrow 1
  \] 
  As we have seen above, the group
  $\Gamma \cap N$ is a lattice in $N$ and polycyclic. Thus preimages of virtually cyclic
  subgroups are virtually polycyclic and hence they lie in \(\mathfrak{FJ}\)
  by Theorem~\ref{the:FJinh}~\eqref{the:FJinh:solv} and~\eqref{the:FJinh:over}.  If $N$ is
  non-trivial, it has dimension bigger than zero and thus $G/N$ has smaller
  dimension. Hence the lattice $\Gamma/(\Gamma\cap N)\rtimes \mathbb{Z}$ lies in
  \(\mathfrak{FJ}\) by induction assumption and so does \(\Gamma \rtimes_\varphi \Z\) by Theorem~\ref{the:FJinh}\,\eqref{the:FJinh:hom}.  If $N$ is trivial, then the radical $R$ is
  also trivial which means that $G$ is additionally semisimple.

  Let $K$ be the maximal compact factor of $G$.  It is likewise characteristic and hence
  $\varphi$-invariant.  Thus we have a short exact sequence
  \[ 
  1 \rightarrow \Gamma \cap K \rightarrow \Gamma\rtimes_\varphi \mathbb{Z}
  \stackrel{p}{\rightarrow} \Gamma/(\Gamma\cap K)\rtimes_\varphi \mathbb{Z} \rightarrow  1. 
  \] 
  Since $K\cap \Gamma$ is finite, any preimage of a virtually cyclic subgroup under
  $p$ is again virtually cyclic and hence it lies in \(\mathfrak{FJ}\).  Thus by Theorem~\ref{the:FJinh}\,\eqref{the:FJinh:hom} it remains
  to show the statement for $\Gamma/(\Gamma\cap K)\rtimes_\varphi \mathbb{Z}$.  As
  explained above $\Gamma/(\Gamma \cap K)$ is a lattice in the Lie group $G/K$. If $K$ is
  non-trivial, then the target has smaller dimension and thus satisfies the Farrell--Jones
  conjecture by induction assumption. It remains to consider the case where $G$ is
  semisimple without compact factors.

  The outer automorphism group of $G$ is finite by Lemma~\ref{lemma:OutG}.  Since
  $\Gamma\rtimes_{\varphi^n} \Z$ has finite index in $\Gamma \rtimes_\varphi \Z$, we can
  use Theorem~\ref{the:FJinh}~\eqref{the:FJinh:over} to replace $\varphi$ by a power of
  $\varphi$ and thus we may assume that $\varphi$ is given by conjugation with $g\in
  G$. By~\cite{Greenleaf-Moskowitz:Finiteness}*{Corollary~2.2, p.~313} we have that the
  Weyl group $\{g\in G|g\Gamma g^{-1}=\Gamma\}/\Gamma$ is finite. Thus after further
  passing to a power we may assume that $\varphi$ is given by conjugation with $\gamma \in
  \Gamma$. Thus $\Gamma\rtimes_\varphi \Z\cong \Gamma\times \Z$. The isomorphism is given
  by the identity on $\Gamma$ and it sends the generator of $\Z \subset \Gamma
  \rtimes_\varphi \Z$ to $(\gamma,1)\in \Gamma \times \Z$.  Hence $\Gamma \rtimes_\varphi
  \Z$ lies in \(\mathfrak{FJ}\) by Theorem~\ref{the:FJinh}~\eqref{the:FJinh:prod} 
   and Proposition~\ref{prop:FJsemisimp}.
\end{proof}

\begin{remark} We did not show the Farrell--Jones conjecture for all groups of the form
  $\Gamma\rtimes \mathbb{Z}$, where $\Gamma$ is a lattice in a connected Lie group.  We
  only showed this for those automorphisms which extend to an automorphism of the
  surrounding Lie group.  Nevertheless, in interesting cases these extensions always exist
  and are unique, most notably if \(\Gamma\) is an irreducible lattice in a connected
  semisimple Lie group without compact factors, with trivial center and not isomorphic to
  \(\textup{PSL}(2, \R)\)~\cite{Margulis:Discrete}*{Theorem~7.5 and Remark~7.6, p.~254}.
  Note that the last requirement is essential because the group \(\textup{PSL}(2, \Z)\)
  contains the free group on three letters \(F_3\) as a subgroup of index twelve, so
  \(F_3\) is a lattice in \(\textup{PSL}(2, \R)\). If unique extension of automorphisms
  held for lattices in \(\textup{PSL}(2, \R)\), we would obtain an embedding
  \(\textup{Aut}(F_3) \rightarrow \textup{Aut}(\mathfrak{sl}(2, \R)) \subset
  \textup{GL}(3, \R)\).  It is however well-known that \(\textup{Aut}(F_3)\) has no
  faithful linear representation~\cite{Formanek-Procesi:AutFnnotlinear}.
\end{remark}

\begin{theorem}\label{thm:FJpi0}
  Let $\Gamma$ be a lattice in a Lie group $G$.  If $\pi_0(G)$ lies in \(\mathfrak{FJ}\),
  so does $\Gamma$.
\end{theorem}

\begin{proof}
  Let $p:G\rightarrow \pi_0(G)$ be the projection.  Its kernel $G^0$ is the path component
  of the identity.  We get an induced group homomorphism $\Gamma\rightarrow p(\Gamma)$
  whose target lies in \(\mathfrak{FJ}\).  We want to apply
  Theorem~\ref{the:FJinh}~\eqref{the:FJinh:hom}.  Since \(G^0\) is open, the path
  components of \(G/\Gamma\) have positive measure.  Therefore the index of \(p(\Gamma)\)
  in \(\pi_0(G)\) must be finite.  Since \(G^0\) is closed,~\cite{Vinberg-et-al:Discrete}*{Theorem~4.7, p.~23} 
  asserts that \(\Gamma \cap G^0\) is
  a lattice in \(G^0\). (Mind the footnote on p.~\pageref{footnote}!) So the group
  $\Gamma\cap G^0$ lies in \(\mathfrak{FJ}\) by Theorem~\ref{thm:mainthm}.  Thus it
  remains to check that the preimage of every infinite cyclic subgroup $Z$ of $p(\Gamma)$
  lies in \(\mathfrak{FJ}\).  We have short exact sequences
  \[
  \xymatrix{1\ar[r]& G^0\ar[r]& p^{-1}(Z)\ar[r]& Z\ar[r]& 1\\
    1\ar[r]& \Gamma \cap G^0\ar[r]\ar@{^{(}->}[u]& p^{-1}(Z)\cap \Gamma
    \ar[r]\ar@{^{(}->}[u]& Z\ar[r]\ar@{=}[u]& 1.}
  \]
  Thus $p^{-1}(Z)\cap \Gamma$ can be expressed as a semidirect product $G^0\cap \Gamma
  \rtimes_\varphi \mathbb{Z}$. The automorphism $\varphi$ is given by conjugation with a
  preimage $\gamma \in p^{-1}(Z) \cap \Gamma$ of a generator of $Z$. Thus we can apply
  Proposition~\ref{prop:FJconnauto} to the lattice $\Gamma \cap G^0$ in $G^0$ with the
  automorphism $\varphi$.
\end{proof}

\begin{proof}[Proof of Theorem~\ref{thm:mostgeneral}.]
  Let \(K\) be the unique maximal, compact, normal subgroup of \(G^0\).  The Montgomery--Zippin solution to Hilbert's fifth problem implies that the factor group \(G^0/K\) is a Lie group, see for instance~\cite{Bagley-Peyrovian:CompactSubgroups}*{Lemma~1, p.~274}.  By uniqueness $K$ is characteristic in $G^0$ and
  thus normal in $G$. Consequently $G/K$ is homeomorphic to the disjoint union of
  $\pi_0(G)$-copies of $G^0/K$. The discrete space \(\pi_0(G)\) is countable because \(G\)
  is second countable, so in fact \(G/K\) is a Lie group as well.  Moreover \(\pi_0(G/K) =  \pi_0(G)\).  
  As in the beginning of the proof of Proposition~\ref{prop:FJsemisimp}, the
  group \(\Gamma / \Gamma \cap K\) is a lattice in \(G/K\) and \(\Gamma / \Gamma \cap K\)
  lies in \(\mathfrak{FJ}\) if and only if \(\Gamma\) does.  Theorem~\ref{thm:FJpi0}
  completes the proof.
\end{proof}

\end{document}